\documentclass[11pt,letterpaper]{amsart}
\usepackage[usenames,dvipsnames]{color}
\usepackage{amsthm,amsfonts,amssymb,amsmath,amsxtra, appendix}
\usepackage[all]{xy}
\SelectTips{cm}{}

\usepackage{verbatim}
\usepackage[margin=1.25in]{geometry}
\usepackage{mathrsfs}
\usepackage{diagbox}
\usepackage{enumerate}
\usepackage[shortlabels]{enumitem}
\usepackage{tkz-euclide}
\usepackage{setspace}
\usepackage{hyperref}
\usepackage{csquotes}
\usepackage{xspace}
\usepackage{etoolbox}
\usepackage{varwidth}
\usepackage{tensor}
\usepackage{mathtools}
\usepackage{longtable}
\usepackage{multirow}

\usepackage{tikz}
\usepackage{tikz-cd}

\def\ge{\geqslant}
\def\le{\leqslant}
\def\a{\alpha}
\def\b{\beta}
\def\g{\gamma}

\def\D{\Delta}
\def\L{\Lambda}

\def\s{\sigma}
\def\t{\tau}
\def\th{\theta}
\def\k{\kappa}
\def\l{\lambda}

\def\pr{\mathrm{pr}}
\def\aff{\mathrm{aff}}

\def\<{\langle}
\def\>{\rangle}

\newcommand{{\BG}}{\ensuremath{\mathbb {G}}\xspace}

\newcommand{{\BK}}{\ensuremath{\mathbb {K}}\xspace}

\newcommand{\BR}{\ensuremath{\mathbb {R}}\xspace}
\newcommand{\BS}{\ensuremath{\mathbb {S}}\xspace}

\newcommand{\BZ}{\ensuremath{\mathbb {Z}}\xspace}

\newcommand{\CF}{\ensuremath{\mathcal {F}}\xspace}

\newcommand{\CP}{\ensuremath{\mathcal {P}}\xspace}

\DeclareMathOperator{\Adm}{Adm}

\def\tW{\tilde W}
\def\ta{\tilde \a}
\def\tb{\tilde \b}
\def\tg{\tilde \g}

\def\tS{\tilde \BS}

\DeclareMathOperator{\supp}{supp}

\newtheorem{theorem}{Theorem}

\newtheorem{proposition}[theorem]{Proposition}
\newtheorem{lemma}[theorem]{Lemma}

\newtheorem{corollary}[theorem]{Corollary}

\theoremstyle{definition}
\newtheorem{definition}[theorem]{Definition}
\newtheorem{example}[theorem]{Example}
\newtheorem*{example*}{Example}

\newtheorem{remark}[theorem]{Remark}

\newtheorem*{function*}{Function}

\numberwithin{equation}{section}
\numberwithin{theorem}{section}

\setitemize[0]{leftmargin=*,itemsep=\the\smallskipamount}
\setenumerate[0]{leftmargin=*,itemsep=\the\smallskipamount}

\renewcommand{\to}{%
   \ifbool{@display}{\longrightarrow}{\rightarrow}%
   }
\let\shortmapsto\mapsto
\renewcommand{\mapsto}{%
   \ifbool{@display}{\longmapsto}{\shortmapsto}%
   }
\newlength{\olen}
\newlength{\ulen}
\newlength{\xlen}
\newcommand{\xra}[2][]{%
   \ifbool{@display}%
      {\settowidth{\olen}{$\overset{#2}{\longrightarrow}$}%
       \settowidth{\ulen}{$\underset{#1}{\longrightarrow}$}%
       \settowidth{\xlen}{$\xrightarrow[#1]{#2}$}%
       \ifdimgreater{\olen}{\xlen}%
          {\underset{#1}{\overset{#2}{\longrightarrow}}}%
          {\ifdimgreater{\ulen}{\xlen}%
             {\underset{#1}{\overset{#2}{\longrightarrow}}}
             {\xrightarrow[#1]{#2}}}}%
      {\xrightarrow[#1]{#2}}
   }
\makeatother
\newcommand{\xyra}[2][]{%
   \settowidth{\xlen}{$\xrightarrow[#1]{#2}$}%
   \ifbool{@display}%
      {\settowidth{\olen}{$\overset{#2}{\longrightarrow}$}%
       \settowidth{\ulen}{$\underset{#1}{\longrightarrow}$}%
       \ifdimgreater{\olen}{\xlen}%
          {\mathrel{\xymatrix@M=.12ex@C=3.2ex{\ar[r]^-{#2}_-{#1} &}}}%
          {\ifdimgreater{\ulen}{\xlen}%
             {\mathrel{\xymatrix@M=.12ex@C=3.2ex{\ar[r]^-{#2}_-{#1} &}}}
             {\mathrel{\xymatrix@M=.12ex@C=\the\xlen{\ar[r]^-{#2}_-{#1} &}}}}}%
      {\mathrel{\xymatrix@M=.12ex@C=\the\xlen{\ar[r]^-{#2}_-{#1} &}}}%
   }
\makeatletter
\newcommand{\xla}[2][]{%
   \ifbool{@display}%
      {\settowidth{\olen}{$\overset{#2}{\longleftarrow}$}%
       \settowidth{\ulen}{$\underset{#1}{\longleftarrow}$}%
       \settowidth{\xlen}{$\xleftarrow[#1]{#2}$}%
       \ifdimgreater{\olen}{\xlen}%
          {\underset{#1}{\overset{#2}{\longleftarrow}}}%
          {\ifdimgreater{\ulen}{\xlen}%
             {\underset{#1}{\overset{#2}{\longleftarrow}}}
             {\xleftarrow[#1]{#2}}}}%
      {\xleftarrow[#1]{#2}}
   }
\newcommand{\isoarrow}{%
   \ifbool{@display}{\overset{\sim}{\longrightarrow}}{\xrightarrow\sim}%
   }

\begin{document}

\title[]{On the Face Map of the Admissible Set With Parahoric Level} 

\author[Qingchao Yu]{Qingchao Yu}
\address{Institute for Advanced Study, Shenzhen University, Nanshan District, Shenzhen, Guangdong, China}
\email{qingchao\_yu@outlook.com}

%\thanks{}
%\keywords{Admissible set, Face Map}

%\date{\today}

\begin{abstract}
To each face $\CF$ of the coweight polytope $\CP_{\mu}$, we associate a subset $\Adm(\mu)_{\CF}$ of the $\mu$-admissible set $\Adm(\mu)$, which we refer to as a face of $\Adm(\mu)$. This gives rise to a face decomposition of $\Adm(\mu)$. As an application, we give a complete description of the fibers of the face map $\nabla$ defined by Pappas--Rapoport \cite{PR26} and prove that the face map is surjective.
\end{abstract}

\maketitle

%\tableofcontents

\section*{Introduction}
\subsection{The admissible set}
Let $(\Phi, X^*,\Phi^{\vee},X_* ,\Pi_0)$ be a reduced based root datum. Let $W_0$ be the Weyl group. Let $\tW = X_* \rtimes W_0 = \{t^{\l}z \mid \l \in X_*, z \in W_0\}$ be the extended affine Weyl group. Let $\mu$ be a (not necessarily minuscule) dominant coweight. The coweight polytope $\CP_{\mu}$ is defined as the convex hull of $W_0(\mu)$, the $W_0$-orbit of $\mu$. Let $\mathscr{F}(\CP_{\mu})$ be the poset of (closed) faces of $\CP_{\mu}$, where $\CF \le \CF'$ if $\CF \subset \CF'$. Here, a face of a polytope is the polytope itself or the intersection of the polytope with a supporting hyperplane.

Define the $\mu$-admissible set as 
$$\Adm(\mu) = \{w \in \tW \mid w \le t^{\mu'}\text{ for some } \mu' \in W_0(\mu)\}.$$

The notion of the admissible set arose from the study of local models of Shimura varieties and moduli of $G$-Shtukas. We refer to \cite{PRS} and the references therein for the early history of the subject, and \cite{PZ13}, \cite{SW20}, \cite{HR23}, and \cite{AGLR} for more recent developments. The special fibers of local models admit stratifications by Schubert cells parametrized by the admissible sets (see \cite[Theorem 9.3]{PZ13}). The Bernstein functions in the center of the affine Hecke algebras are supported on the admissible sets (see \cite[Corollary 5.7]{Ha01}).

The combinatorics of $\Adm(\mu)$ has been studied intensively. Haines--He \cite{HH17} proved that $\Adm(\mu)$ is equal to the intersection of combinatorial obtuse cones, which indicates that $\Adm(\mu)$ should be regarded as the alcove-theoretic analogue of the coweight polytope $\CP_{\mu}$. Mili\'cevi\'c--Viehmann \cite{MV20}, He--Yu \cite{HY}, and Schremmer \cite{Sch24} used the quantum Bruhat graph to study $\Adm(\mu)$ and gave a description of $\Adm(\mu)$.

\subsection{The divisor map and the face map}
Very recently, under the hypotheses of \cite[\S4.1]{PR26}, Pappas--Rapoport
defined a morphism called the \emph{divisor map} $\D$ on the
quotient stack of the local model. For simplicity, we restrict ourselves to the Iwahori case in this introduction. The divisor map induces a map 
\begin{align*}
   |\D|: \Adm(\mu) \longrightarrow \mathscr{F}(\CP_{\mu}).
\end{align*}
Next, they defined the \emph{face map} 
\begin{align*}
   \nabla: \Adm(\mu) \longrightarrow \mathscr{F}(\CP_{\mu})
\end{align*}
as follows. For any $w \in \Adm(\mu)$, define $\L(w) = \{\mu' \in W_0(\mu) \mid w\le t^{\mu'}\}$ and define $\nabla(w)$ as the smallest face of $\CP_{\mu}$ containing $\L(w)$. Then it is clear that $|\D|(w)\supset \nabla(w)$ for any $w \in \Adm(\mu)$. They proved that the map $|\D|$ coincides with the map $\nabla$ for general parahoric level (cf. \cite[Theorem 4.1.7]{PR26}). They had also conjectured that the face map $\nabla$ is surjective.

%They prove that (i) $|\D|$ is order-reversing, (ii) $|\D|$ sends the unique minimal element $\t_{\mu}$ in $\Adm(\mu)$ to the maximal face $\CP_{\mu}$, and (iii) for any $\mu'\in W_0(\mu)$, $|\D|$ sends $t^{\mu'}$ to the zero-dimensional face $\{\mu'\}$ of $\CP_{\mu}$ (cf. \cite[Proposition 4.1.5]{PR26}).    

The goal of this paper is to give an explicit description of the fibers of the face map $\nabla$ and to prove that $\nabla$ is surjective for general parahoric levels.

\subsection{The main results}
To each face $\CF$ of $\CP_{\mu}$, we associate a subset $\Adm(\mu)_{\CF}$ of $\Adm(\mu)$, which we call a \emph{face} of $\Adm(\mu)$. This is defined as the admissible set in a certain Levi subgroup. We refer to \S\ref{sec:adm} and \S\ref{sec:main} for the exact definition. Moreover, we define the notions of the \emph{interior}, the \emph{boundary}, and the \emph{center} of the face $\Adm(\mu)_{\CF}$ (see \S\ref{sec:main}). These notions play important roles in analyzing the face structure of $\Adm(\mu)$ and the face map.

The main results in the Iwahori case are summarized as follows. We refer to Theorems~\ref{thm:main1} and~\ref{thm:main2} for the precise statements in the Iwahori case and to Theorem~\ref{thm:main3} for the parahoric case.

\begin{theorem}\label{thm:01}
The following statements hold.
\begin{enumerate}
\item Let $\CF$ be a face of the polytope $\CP_{\mu}$. We have
\begin{align*}
    \Adm(\mu)_{\CF}
    = \{w \in \Adm(\mu) \mid \L(w) \subset \CF\}.
\end{align*}

\item For any face $\CF$ of $\CP_{\mu}$, the fiber $\nabla^{-1}(\CF)$ is equal to $\Adm(\mu)_{\CF}^{\circ}$, the interior of the face $\Adm(\mu)_{\CF}$, which is defined as the complement of the union of all smaller faces $\Adm(\mu)_{\CF'}$ in $\Adm(\mu)_{\CF}$.

\item There is a decomposition
\begin{align*}
    \Adm(\mu)
    = \bigsqcup_{\CF} \Adm(\mu)_{\CF}^{\circ}.
\end{align*}
Here, $\CF$ runs over all faces of $\CP_{\mu}$.

\item The face map $\nabla$ is surjective.
\end{enumerate}
\end{theorem}

We note that the decomposition in Theorem~\ref{thm:01} (3) is closely related to the Hodge--Newton decomposition
\begin{align*}
    \Adm(\mu, b)
    = \bigsqcup_{P' = M' N' \in \mathscr{P}^{\s}}
      \Adm^{M'}(z_{P'}(\mu),b_{P'})
\end{align*}
in \cite[Theorem 4.11 (1)]{GHN2}. Note also that the interior $\Adm(\mu)^{\circ}$ of $\Adm(\mu)$ is closely related to the Hodge--Newton indecomposable Newton strata.

The paper is organized as follows. In \S\ref{sec:polytope}, we introduce the coweight polytope $\CP_{\mu}$ and recall Maxwell's description of the face structure. In \S\ref{sec:adm}, we introduce the $\mu$-admissible set and define its faces. In \S\ref{sec:main}, we state our main results in the Iwahori case, and the proof of Theorem~\ref{thm:main1} is given in \S\ref{sec:proof}. In the final two sections, we extend the results to the parahoric case and prove Theorem~\ref{thm:main3}.

{\bf Acknowledgements:} I would like to thank Michael Rapoport for introducing the problem to me and for many helpful discussions and suggestions. I would like to thank Xuhua He for his constant support of my research and for many helpful discussions. I would like to thank Hongsheng Hu, Tao Gui, Felix Schremmer, Sian Nie, and Pengcheng Li for their helpful discussions. The author is partially supported by the National Natural Science Foundation of China (grant no. 12501018). This manuscript was prepared with the assistance of AI tools to improve linguistic readability.

\section{The Polytope
\texorpdfstring{$\CP_{\mu}$}{P(mu)}}\label{sec:polytope}

Let $(\Phi, X^*,\Phi^{\vee},X_* ,\Pi_0)$ be a reduced based root datum. Let $\<-,-\>$ be the natural pairing between $X_\ast$ and $X^\ast$. Let $W_0$ be the Weyl group. Let $\Phi^+$ and $\Phi^-$ be the sets of positive and negative roots, respectively. Let $V = X_* \otimes_{\BZ} \BR$. Then $W_0$ acts on $V$ naturally. A coweight $\l \in X_*$ is said to be dominant if $\<\l,\a\>\ge0$ for any $\a\in\Phi^+$. For any $I\subset\Pi_0$, we denote by $W_I$ the subgroup of $W_0$ generated by $\{s_{\a} \mid \a\in I\}$ and denote by $W^I$ the set of minimal length representatives in the left cosets of $W_I$.

In the rest of this paper, we fix a (not necessarily minuscule) dominant coweight $\mu$. Let $J_{\mu} = \{\a\in\Pi_0 \mid \<\mu,\a\> = 0\}$.

A \emph{polytope} is the convex hull of finitely many points in a Euclidean space. A subset $\CF$ of a polytope $\CP$ is called a (closed) \emph{face} of $\CP$ if $\CF = \CP$ or $\CF$ is the intersection of $\CP$ with a supporting hyperplane of $\CP$.

The \emph{coweight polytope} $\CP_{\mu}$ is defined as the convex hull of $W_0(\mu)$ in $V$. Here, $W_0(\mu)$ is the $W_0$-orbit of $\mu$. Let $\mathscr{F}(\CP_{\mu})$ be the poset of nonempty (closed) faces of $\CP_{\mu}$, where $\CF \le \CF'$ if and only if $\CF \subset \CF'$. %Note that $W_0$ acts on $\mathscr{F}(\CP_{\mu})$ naturally. 

We now recall a description of the face poset $\mathscr{F}(\CP_{\mu})$ given by Maxwell \cite{MAX89}. Relevant results can be found in \cite[Theorem 1]{KhRi12} and \cite[Theorem 2.2]{kh13}.

For $I \subset \Pi_0$ and $a \in W^I$, we denote by $\CF_{a, I}$ the convex hull of $aW_I(\mu) = \{a u (\mu) \mid u \in W_I\}$. Note that $\CF_{a, I}$ only depends on the set $a W_I W_{J_{\mu}}$. In other words, for $I,I'\subset \Pi_0$, $a \in W^I$ and $a'\in W^{I'}$, if $a W_I W_{J_{\mu}} = a' W_{I'} W_{J_{\mu}}$, then $\CF_{a, I} = \CF_{a', I'}$.

\begin{theorem}[{\cite[Section 5]{MAX89}}]\label{thm:Maxwell}
\begin{enumerate}
\item For $I\subset\Pi_0$ and $a\in W^I$, the set $\CF_{a, I}$ is a face of $\CP_{\mu}$ whose set of vertices is $a W_I(\mu)$.

\item There is an isomorphism of posets:
\begin{align*}
    \{a W_I W_{J_{\mu}} \mid I\subset\Pi_0, a\in W^I\} &\xlongrightarrow{\sim} \mathscr{F}(\CP_{\mu})\\
    a W_I W_{J_{\mu}} &\mapsto  \CF_{a, I}.
\end{align*}

\item Let $I,I' \subset\Pi_0$ and $a\in W^I, a' \in W^{I'}$. If $a W_I W_{J_{\mu}} \cap a' W_{I'} W_{J_{\mu}} =\emptyset$, then $\CF_{a, I} \cap \CF_{a', I'} = \emptyset$. If $a W_I W_{J_{\mu}} \cap a' W_{I'} W_{J_{\mu}} \ne \emptyset$, then $a W_I W_{J_{\mu}} \cap a' W_{I'} W_{J_{\mu}} = a'' W_{I''} W_{J_{\mu}} $ and $\CF_{a, I} \cap \CF_{a', I'} = \CF_{a'', I''}$ for some $I''\subset \Pi_0$ and $a'' \in W^{I''}$.
\end{enumerate}
\end{theorem}

If $I=\emptyset$, then $\CF_{a, I}$ is the 0-dimensional face $\{a(\mu)\}$. If $I = \Pi_0$, then $\CF_{a, I}$ is the polytope $\CP_{\mu}$ itself. Figure 1 shows the faces of $\CP_{\mu}$ in the case of type $A_2$ and $\mu = 2\omega_1^{\vee} = (2,0,0)$. Figure 2 shows the faces of $\CP_{\mu}$ in the case of type $C_2$ and $\mu = \omega_1^{\vee} = (1,0)$.

%\newpage
\begin{figure}[htbp]
\centering
\begin{minipage}{0.45\textwidth}
\centering
\begin{tikzpicture}[scale=0.6]

\coordinate (T) at (0, 2.6);  
\coordinate (B) at (0, -2.6); 
\coordinate (R) at (4.5, 0);

\draw[ultra thick] (T) -- (R) -- (B) -- cycle;

\coordinate (MidCenter) at (1.5, 0);
\node at (MidCenter) {\scriptsize $\CF_{1, \{\a_1,\a_2\}}$};

\node[above, yshift=2mm] at (T) {\scriptsize $\CF_{s_{1}, \emptyset}$};
\node[below, yshift=-2mm] at (B) {\scriptsize $\CF_{s_2 s_1, \emptyset}$};
\node[right, xshift=2mm] at (R) {\scriptsize $\CF_{1, \emptyset}$};

\coordinate (MidL) at (0.2, 0);
\coordinate (MidUR) at (2.1, 1.1);
\coordinate (MidBR) at (2.1, -1.1);

\node[left, xshift=-2mm] at (MidL) {\scriptsize $\CF_{s_1 s_2, \{\a_1\}}$};
\node[above right, xshift=1mm, yshift=1mm] at (MidUR) {\scriptsize $\CF_{1, \{\a_1\}}$};
\node[below right, xshift=1mm, yshift=-1mm] at (MidBR) {\scriptsize $\CF_{s_{2}, \{\a_1\}}$};

\end{tikzpicture}
\caption{}
\end{minipage} 
\begin{minipage}{0.45\textwidth}
\centering

\begin{tikzpicture}[scale=1.8]
\coordinate (L) at (-1, 0);   
\coordinate (T) at (0, 1);  
\coordinate (R) at (1, 0);   
\coordinate (B) at (0, -1); 
\coordinate (MidBL) at (-1/2, -1/2); 
\coordinate (MidBR) at (1/2, -1/2);
\coordinate (MidTR) at (1/2, 1/2); 
\coordinate (MidTL) at (-1/2, 1/2);

\draw[ultra thick] (T) -- (R) -- (B) -- (L) -- cycle;

\node at (0, 0) {\scriptsize $\mathcal{F}_{1, \{\a_1,\a_2\}}$};

\node[left] at (L)    {\scriptsize $\mathcal{F}_{s_1s_2s_1, \emptyset}$};
\node[below] at (B)   {\scriptsize $\mathcal{F}_{s_2s_1, \emptyset}$}; 
\node[right] at (R)   {\scriptsize $\mathcal{F}_{1, \emptyset}$};
\node[above] at (T)   {\scriptsize $\mathcal{F}_{s_1, \emptyset}$};  
\node[below left] at (MidBL) {\scriptsize $\mathcal{F}_{s_2s_1s_2, \{\a_1\}}$}; 
\node[below right] at (MidBR) {\scriptsize $\mathcal{F}_{s_2, \{\a_1\}}$};   
\node[above right] at (MidTR) {\scriptsize $\mathcal{F}_{1, \{\a_1\}}$};     
\node[above left] at (MidTL) {\scriptsize $\mathcal{F}_{s_1s_2, \{\a_1\}}$};  

\end{tikzpicture}
\caption{}
\end{minipage}
\end{figure}

For a face $\CF$ of $\CP_{\mu}$, there may be more than one pair $(a, I)$ with $\CF = \CF_{a, I}$. By \cite[Sections 2]{MAX89}, one can choose $(a, I)$ with $\CF = \CF_{a, I}$ such that there is no connected component of $I$ lying in $J_{\mu}$ (see also \cite[Corollary 1.3]{Ren09} and \cite[Proposition 3.2]{Vin91}). In this case, the dimension of $\CF_{a, I}$ is equal to the cardinality of $I$. We do not need this result in this paper, but it is helpful for the reader to keep in mind that one can choose $I$ in this way.

\section{The Admissible Set $\Adm(\mu)$}\label{sec:adm}

%The set of positive affine roots is defined as $\tilde\Phi^+ = (\Phi^+ \times \BZ_{\ge0} )\sqcup (\Phi^- \times \BZ_{\ge1})$, and its complement is the set of negative affine roots. 

The \emph{extended affine Weyl group} is defined as
$$\tW = X_*\rtimes W_0 = \{ t^{\l}z \mid \l \in X_*, z\in W_0\}.$$
The set of simple affine reflections is $\tilde{\BS} = \{s_\a \mid \a \in \Pi_0 \} \sqcup \{ t^{\th^{\vee}}s_{\th} \}$, where $\th$ runs over the highest roots of the connected components of $\Phi$ and $\th^{\vee}$ is the coroot dual to $\th$.

%Let $\ell$ be the length function
%\begin{align*}
%    \ell(t^\l z) = \sum_{\a \in \Phi^+,z^{-1}(\a) \in \Phi^+} |\< \l,\a\> | + \sum_{\a \in \Phi^+,z^{-1}(\a) \in \Phi^-} |\< \l,\a\> - 1|. 
%\end{align*}

%By convention, we view $\Phi$ as a subset of $\tilde\Phi$ via the embedding $\a \to (\a,0)$. In particular, $\Delta_0$ is viewed as a subset of $\tilde\Delta$. We also view $W_0$ as a natural subset of $\tW$.

%Note that $\tW$ acts on For any $w = t^{\l}z \in   \tW$, the action on $\tilde\Phi$ is given by $$\tilde{\a} = (\a,k)\in \Phi\times\BZ\mapsto w(\a,k) = (z(\a), k-\<\l,z(\a)\>).$$ The affine reflection corresponding to $\tilde\a \in \tilde\Phi^+$ is $s_{\tilde{\a}} = s_{\a} t^{k\a^{\vee}} \in \tW$. 

Let $W_{\aff}$ be the affine Weyl group, that is, the subgroup of $\tW$ generated by $\tS$. Then $(W_{\aff},\tS)$ is a Coxeter system. It is clear that $W_{\aff} = \BZ\Phi^{\vee} \rtimes W_0$ and $\tW = W_{\aff} \rtimes \Omega$, where $\Omega$ is the subset of length-zero elements in $\tW$. Let $\leq$ and $\ell$ be the Bruhat order and the length function of the Coxeter group $W_{\aff}$, respectively. We extend $\leq$ to a partial order on $\tW$ as follows. For $w_1,w_2 \in W_{\aff}$ and $\t_1, \t_2 \in \Omega$, define $w_1\t_1 \le w_2\t_2$ if $\t_1 = \t_2$ and $w_1\le w_2$. We extend $\ell$ to $\tW$ by defining $\ell(w\t) = \ell(w)$ for $w \in W_{\aff}$ and $\t\in \Omega$.

The $\mu$-admissible set of $\tW$ is defined as
\begin{align*}
    \Adm(\mu) = \{w \in \tW \mid w \le t^{\mu'}\text{ for some } \mu' \in W_0(\mu)\}.
\end{align*}
The maximal elements in $\Adm(\mu)$ are $t^{\mu'}$, $\mu' \in W_0(\mu)$. Note that $\ell(t^{\mu'}) = \<\mu,2\rho\>$ for any $\mu' \in W_0(\mu)$. Here, $\rho = \frac{1}{2}\sum_{\a\in\Phi^+}\a$.

It is convenient to think of $\Adm(\mu)$ as the alcove-theoretic analogue of the coweight polytope $\CP_{\mu}$, which we now explain.

The set of \emph{affine roots} is $\tilde\Phi= \Phi \times \BZ$. For each affine root $\ta=(\a,k)\in \tilde\Phi$, the associated root hyperplane is $H_{\ta} = \{ v \in V \mid k+ \<v,\a\> =0   \}$ and the associated affine reflection is $s_{\ta} = s_{\a} t^{k\a^{\vee}}$, where $\a^{\vee}$ is the coroot dual to $\a$.

By definition, \emph{alcoves} are connected components of $V - \bigcup_{\tilde{\a}}H_{\tilde{\a}}$, where $\tilde{\a}$ runs over the set of affine roots $\tilde\Phi$. Let $\Sigma$ be the set of alcoves. Note that $\tW$ acts on $V$ via affine transformation $(t^\l z,v)\mapsto z(v) +\l$. Then $\tW$ acts on $\Sigma$ naturally. The base alcove is defined as 
$$\mathfrak{a} = \{ v\in V\mid 0 < \<v,\a \> < 1\text{ for every }\a\in \Phi^+\}.$$
There is a natural bijection between $W_{\aff}$ and $\Sigma$ sending $w$ to $w(\mathfrak{a})$. Note that for $w\in\tW$, the length $\ell(w)$ is the number of root hyperplanes $H_{\ta}$ separating $\mathfrak{a}$ and $w(\mathfrak{a})$. In particular, the set of length-zero elements $\Omega$ is the stabilizer of the base alcove $\mathfrak{a}$.

For any alcove $\mathfrak{b}$, we define the Bruhat order $\le_{\mathfrak{b}}$ on $\Sigma$ with respect to $\mathfrak{b}$ as follows. Let $\mathfrak{b}', \mathfrak{b}'' \in \Sigma$. Define $\mathfrak{b}' \le_{\mathfrak{b}} \mathfrak{b}''$ if there exists a sequence of alcoves $\mathfrak{b}_1 = \mathfrak{b}'$, $\mathfrak{b}_2, \ldots, \mathfrak{b}_{r+1} = \mathfrak{b}''$, and affine roots $\ta_1,\ta_2,\ldots,\ta_{r}$ such that $\mathfrak{b}_{i+1} = s_{\ta_i}(\mathfrak{b}_i)$ for each $i$, and $\mathfrak{b}_i$ and $\mathfrak{b}$ are on the same side of the hyperplane $H_{\ta_i}$. For elements $w,w'$ in the same coset $W_{\aff}\t$, $\t\in\Omega$, we have $w(\mathfrak{a}) \le_{\mathfrak{a}} w' (\mathfrak{a}) \text{ if and only if } w \le w'$.

For any root hyperplane $H$ and $z \in W_0$, let $H^{z-} = H + z(C^-)$ be the connected component of $V - H$ that contains all sufficiently
deep alcoves in the Weyl chamber $z(C^-)$. Here, $C^- =\{v \in V\mid \<v, \a\> < 0\text{ for every }\a \in \Phi^+\}$ is the anti-dominant Weyl chamber.

Let $w'\in\tW$ and $z\in W_0$. Following \cite[Definition 1.1]{HH17}, we define the \emph{combinatorial obtuse cone} $O(w',z)$ as the set of $w\in W_{\aff}w'$ such that $w(\mathfrak{a}) \le_{\mathfrak{b}} w'(\mathfrak{a})$ for any alcove $\mathfrak{b}$ sufficiently deep inside the Weyl chamber $z(C^-)$. Equivalently, $O(w',z)$ is the set of $w\in W_{\aff}w'$ such that there exists a sequence of alcoves from $w'(\mathfrak{a})$ to $w(\mathfrak{a})$ in the $z(\Phi^-)$-direction, that is, a sequence of alcoves $\mathfrak{b}_1 = w'(\mathfrak{a})$, $\mathfrak{b}_2, \ldots, \mathfrak{b}_{r+1} = w(\mathfrak{a})$, and affine roots $\ta_1,\ta_2,\ldots,\ta_{r}$ with $\mathfrak{b}_{i+1} = s_{\ta_i}(\mathfrak{b}_i)$ and $\mathfrak{b}_{i+1} \subset (H_{\ta_i})^{z-}$ for each $i$.

%Let $I$ be a subset of $\subset\Pi_0$. We denote by $W_I$ the subgroup of $W_0$ generated by $s_\a$, $\a\in I$. We denote by $\Phi_I$ the subroot system of $\Phi$ spanned by $I$. Set $\Phi^{\pm} = \Phi^{\pm}$. Let $W^I = \{a \in W_0 \mid a(\Phi_I^+) \subset \Phi^+\}$ be the set of minimal representatives of the left cosets of $W_I$.

The combinatorial obtuse cone $O(t^{z(\mu)},z)$ should be regarded as an analogue of the \emph{geometric obtuse cone} $z\bigl(\mu - \sum_{\a \in \Pi_0}\BR_{\ge0}\a^{\vee}\bigr)$ in $V$. It is clear that 
\begin{align}\label{eq:conv}
\CP_{\mu} = \bigcap_{z \in W_0} \bigl(z\bigl(\mu - \sum_{\a \in \Pi_0} \BR_{\ge0}\a^{\vee}\bigr)\bigr).
\end{align}
The following theorem, due to Haines--He, is the alcove-theoretic analogue of (\ref{eq:conv}).

\begin{theorem}[{\cite[Theorem 1.3]{HH17}}]\label{thm:HH}
We have
\begin{align*}
   \Adm(\mu) = \bigcap_{z \in W_0} O(t^{z(\mu)},z). 
\end{align*}
\end{theorem}

Let $I \subset \Pi_0$ and $a \in W^I$. We define a subset $\Adm(\mu)_{a, I}$ as follows. 

Let $\Phi_I$ be the subroot system of $\Phi$ spanned by $I$. Let $\Phi_I^{\pm} = \Phi^{\pm} \cap\Phi_I$. Let $\Phi_{a, I}$ be the subroot system of $\Phi$ spanned by $\{a(\a) \mid \a\in I\}$. Let $\D_{a, I} = \{a(\a)\mid \a \in I\}$ and $\Phi_{a, I}^{\pm} = \Phi_{a, I} \cap \Phi^{\pm}$. Since $a \in W^I$, $\D_{a, I}$ is the base of $\Phi_{a, I}$ corresponding to the positive system $\Phi_{a, I}^+$. Let $W_{a, I}$ be the reflection subgroup of $W_0$ generated by $\{s_{\b} \mid \b \in \Phi_{a, I}\}$. Let $\tW_{a, I} = X_*\rtimes W_{a, I}$ be the extended affine Weyl group of the based root datum $(\Phi_{a, I}, X^*, \Phi_{a, I}^{\vee}, X_*, \D_{a, I})$.

Define
\begin{align*}
\Adm(\mu)_{a, I} = \Adm^{\tW_{a, I}}(a(\mu)),
\end{align*}
where $\Adm^{\tW_{a, I}}(a(\mu))$ is the $a(\mu)$-admissible set of $\tW_{a, I}$, which is the set of $w \in \tW_{a, I}$ such that $w \le_{a, I} t^{au(\mu)}$ for some $u \in W_I$. Here, $\le_{a, I}$ is the Bruhat order on $\tW_{a, I}$. By a standard result of Dyer \cite[Corollary 3.4]{Dyer90}, for $w,w'\in \tW_{a, I}$, if $w \le_{a, I} w'$, then $w \le w'$. Hence, $\Adm(\mu)_{a, I}$ is contained in $\Adm(\mu)$. Note that the maximal elements in $\Adm(\mu)_{a, I}$ are $t^{\mu'}$, where $\mu' \in aW_I(\mu)$, the set of vertices of $\CF_{a, I}$. We call $\Adm(\mu)_{a, I}$ a \emph{face} of $\Adm(\mu)$.

%Then, for $I, I' \subset\Pi_0$ and $a, a' \in W^I$, if $\Adm(\mu)_{a, I} = \Adm(\mu)_{a', I'}$, then $\CF_{a, I} = \CF_{a', I'}$. 

If $I=\emptyset$, then $\Adm(\mu)_{a, I} = \{t^{a(\mu)}\}$. If $I = \Pi_0$, then $\Adm(\mu)_{a, I}$ is equal to $\Adm(\mu)$. Figure 3 is an example in the case of type $A_2$ and $\mu = 2\omega_1^{\vee}$. The origin is at the thick dot. The set $\Adm(\mu)$ is the area inside the thick line. The face $\Adm(\mu)_{s_1, \emptyset}$ consists of the alcove shaded in dark gray. The face $\Adm(\mu)_{s_2, \{\a_1\}}$ consists of the alcoves shaded in light gray. Figure 4 is an example in the case of type $C_2$ and $\mu = \omega_1^{\vee}$. The origin is at the thick dot. The set $\Adm(\mu)$ is the area inside the thick line. The face $\Adm(\mu)_{1, \{\a_1\}}$ consists of the alcoves shaded in dark gray. The face $\Adm(\mu)_{s_2s_1s_2, \{\a_1\}}$ consists of the alcoves shaded in light gray. We see from this example that a face of $\Adm(\mu)$ may not even be ``connected". The reader can compare Figures 3 and 4 with Figures 1 and 2.

\begin{figure}[htbp]
\centering
\begin{minipage}{0.45\textwidth}
\centering
\begin{tikzpicture}[scale=0.7]
\draw[fill=lightgray] (-1.1547005383792515, -2) -- (-0.5773502691896257, -1) -- (0.5773502691896257, -1) --(1.1547005383792515,0)--(2.309401076758503,0)--(2.8867513459481287,1)--(3.464101615137755,0)--(2.309401076758503,0)--(1.7320508075688774,-1)--(0.5773502691896257,-1)--(0.0,-2)-- cycle;
\draw[fill=darkgray](-1.1547005383792515,2)--(-0.5773502691896257,3)--(0,2)--cycle;

\begin{scope}
\clip (-3.5, -3.5) rectangle (4.5, 4.5);
\foreach \i in {-4.618802153517006, -3.464101615137755, -2.309401076758503, -1.1547005383792515, 0.0, 1.1547005383792515, 2.309401076758503, 3.464101615137755, 4.618802153517006} {
\draw[black, thin] (\i,0) -- (\i+ 2.309401076758503, 4);
\draw[black, thin] (\i,0) -- (\i-2.309401076758503,-4);
\draw[black, thin] (\i,0) -- (\i-2.309401076758503,4);
\draw[black, thin] (\i,0) -- (\i+2.309401076758503,-4);
}
\foreach \k in {-4,-3,...,4} {
\draw[black, thin] (-10, \k) -- (10, \k);
}
\end{scope}
\draw[line width=1mm]  (-1.1547005383792515, 0) -- (-0.5773502691896257, 1) -- (-1.1547005383792515, 2) -- (-0.5773502691896257, 3) -- (0, 2) -- (1.1547005383792515, 2) -- (1.7320508075688772, 1) --(2.8867513459481287,1)--(3.4641016151377544,0) --(2.309401076758503,0)--(1.7320508075688774,-1)--(0.5773502691896257,-1)--(0,-2)--(-1.1547005383792515,-2)--(-0.5773502691896257,-1) --cycle;
\draw[fill, black] (0,0) circle [radius=.11cm];
\end{tikzpicture}
\caption{}
\end{minipage}%
\hfill
\begin{minipage}{0.45\textwidth}
\centering
\begin{tikzpicture}[scale=0.4]
\draw[fill=darkgray](0,4)--(2,6)--(2,4)--cycle;
\draw[fill=darkgray](4,2)--(6,2)--(6,0)--(4,0)--cycle;
\draw[fill=lightgray](-4,0)--(-2,2)--(-2,0)--cycle;
\draw[fill=lightgray](0,-2)--(0,-4)--(2,-4)--(2,-2)--cycle;

\begin{scope}
\clip (-6.5, -6.5) rectangle (8, 7.5);
\foreach \i in {-8,-6,...,8} {
\draw[black, thin] (-10,\i) -- (10,\i);
\draw[black, thin] (\i,-10) -- (\i,10);
}
\foreach \k in {-4,-3,...,4} {
\draw[black, thin] (-10, {-10 + 4*\k}) -- (10, {10 + 4*\k});
\draw[black, thin] (-10, {10 - 4*\k}) -- (10, {-10 - 4*\k});
}
\end{scope}
\draw[line width=1mm] (0, 0) -- (-4, 0) -- (2, 6) -- (2, 2) -- (6, 2) -- (6, 0) -- (4, 0) --(4,-2)--(2,-2)--(2,-4)--(0,-4)-- cycle;
\draw[fill, black] (0,0) circle [radius=.23cm];
\end{tikzpicture}
\caption{}
\end{minipage}
\end{figure}

%Let $\le_{a, I}$ be the Bruhat order of $\tW_{a, I}$. By a standard result of Dyer \cite[Corollary 3.4]{Dyer90}, for $w,w'\in W_{a, I, \aff}t^{a(\mu)}$, if $w \le_{a, I} w'$, then $w \le w'$.

We now give an alternate description of the face $\Adm(\mu)_{a, I}$, which plays an important role in the proof of Theorem \ref{thm:main1}. Let $W_{a, I, \aff} = \BZ\Phi_{a, I}^{\vee} \rtimes W_{a, I} \subset \tW_{a,I}$ be the affine Weyl group of the based root datum $(\Phi_{a, I}, X^*, \Phi_{a, I}^{\vee}, X_*, \D_{a, I})$. Note that $s_{\tb} \in W_{a, I, \aff}$ for any $\tb \in \Phi_{a, I} \times \BZ$.

\begin{lemma}\label{lem:def2}
Let $I \subset \Pi_0$ and $a \in W^I$. Then 
$$\Adm(\mu)_{a, I} = \Adm(\mu) \cap  W_{a, I, \aff} t^{a(\mu)}.$$
\end{lemma}
\begin{proof} 
The ``$\subset$" direction follows from the definition $\Adm(\mu)_{a, I} = \Adm^{\tW_{a, I}}(a(\mu))$. 

We now prove the ``$\supset$" direction. By Theorem \ref{thm:HH},
\begin{align*}
    \Adm^{\tW_{a, I}}(a(\mu)) = \bigcap_{u \in W_I} O^{\tW_{a, I}}(t^{a u (\mu)}, a u a^{-1}).
\end{align*}
Here, for $u \in W_I$, the set $O^{\tW_{a, I}}(t^{au(\mu)}, a u a^{-1})$ is the combinatorial obtuse cone of $\tW_{a, I}$. In other words, it is the set of $w \in W_{a, I, \aff}t^{au(\mu)}$ such that there is a sequence of alcoves from $t^{au(\mu)}(\mathfrak{a})$ to $w(\mathfrak{a})$ in the $aua^{-1}(\Phi_{a, I}^-)$-direction, that is, a sequence of alcoves $\mathfrak{b}_1 = t^{au(\mu)}(\mathfrak{a})$, $\mathfrak{b}_2, \ldots, \mathfrak{b}_{r+1} = w(\mathfrak{a})$, and affine roots $\tb_1,\tb_2,\ldots,\tb_{r} \in \Phi_{a, I} \times \BZ$ such that $\mathfrak{b}_{i+1} = s_{\tb_i}(\mathfrak{b}_i)$ and $\mathfrak{b}_{i+1} \subset (H_{\tb_i})^{a u a^{-1}-}$ for each $i$.

Let $w \in  \Adm(\mu) \cap  W_{a, I, \aff} t^{a (\mu)} $. It suffices to prove that $w \in O^{\tW_{a, I}}(t^{au(\mu)}, a u a^{-1})$ for any $u \in W_I$. Since $w \in \Adm(\mu)$, $w \in O(t^{au(\mu)}, a u)$ by Theorem \ref{thm:HH}. Then there exists a sequence of elements $w_1 = t^{au(\mu)}, w_2, \ldots, w_{r+1}=w$, and affine roots $\tb_1,\tb_2,\ldots,\tb_{r}$ such that $w_{i+1} = s_{\tb_i}w_{i}$ and $w_{i+1}(\mathfrak{a}) \subset (H_{\tb_i})^{au-}$ for each $i$.

Since $a \in W^I$, $\Phi_I^- = \Phi_I \cap a^{-1}(\Phi^-)$. Then $au(\Phi_I^-) = aua^{-1}(\Phi_{a, I} \cap \Phi^-) = a u a^{-1} (\Phi^-_{a, I})$. Now it remains to prove that each $\tb_i$ lies in $\Phi_{a, I}\times \BZ$.

Take an arbitrary $e \in \mathfrak{a}$. Then for each $i$, $s_{\tb_i}$ sends $w_{i}(e)$ to $w_i(e) - d_i\g_i^{\vee}$
for some $d_i\in\BR_{>0}$ and $\g_i^{\vee} \in au(\Phi^{\vee, +})$. Here, $\Phi^{\vee, +}$ is the set of positive coroots. Then $t^{au(\mu)}(e) - w(e) = d_1\g_1^{\vee} + d_2\g_2^{\vee}+\cdots + d_{r}\g_r^{\vee}$. On the other hand, since $w\in W_{a, I, \aff} t^{a (\mu)}  =  W_{a, I, \aff} t^{au (\mu)} $, $t^{au(\mu)}(e) - w(e) \in au(\BR\Phi_I^{\vee})$. Here, $\Phi_I^{\vee}$ is the subroot system of $\Phi^{\vee}$ spanned by $\{\a^{\vee}\mid \a\in I\}$. This forces all $\g_i^{\vee}$ to lie in $au(\Phi_I^{\vee}) = a(\Phi_I^{\vee})$. Hence, all $\tb_i$ lie in $\Phi_{a, I} \times \BZ$, as desired.
\end{proof}

\section{Main Results}\label{sec:main}

For any $w \in \Adm(\mu)$, define $\L(w) = \{\mu' \in W_0(\mu) \mid w\le t^{\mu'}\}$.

If $w = \t_\mu$ is the length-zero element in $\Adm(\mu)$, then $\L(w) = W_0(\mu)$. If $w = t^{\mu'}$ for some $\mu' \in W_0(\mu)$, then $\L(w) = \{\mu'\}$. In general, the set $\L(w)$ is very complicated.

The set $\L(w)$ plays an important role in the study of the dual shellability of the $\mu$-admissible set in \cite[Section 4]{HY25} and \cite[Section 4]{HSY26}, where we have provided an explicit description of the set $\L(w)$ in terms of the quantum Bruhat graph. But we do not need that description in this paper.

The first main result of this paper is the following.

\begin{theorem}\label{thm:main1}
Let $I\subset \Pi_0$ and $a \in W^I$. Then 
\begin{align*}
\Adm(\mu)_{a, I} = \{w \in \Adm(\mu) \mid \L(w) \subset \CF_{a, I}\}.    
\end{align*}
\end{theorem}
The proof of Theorem \ref{thm:main1} will be given in \S\ref{sec:proof}.

The following result is a direct consequence of Theorem \ref{thm:main1}. Recall that $\CF_{a, I}$ is the face of $\CP_{\mu}$ with the set of vertices $a W_I(\mu)$.

\begin{proposition}\label{prop:face-corr}
Let $I,I' \subset\Pi_0$ and $a \in W^I, a' \in W^{I'}$. The following hold true. 
\begin{enumerate}
\item We have $\Adm(\mu)_{a, I} = \Adm(\mu)_{a', I'}$ if and only if $\CF_{a, I} = \CF_{a', I'}$.

\item If $\CF_{a, I} \cap \CF_{a', I'} =\emptyset$, then $\Adm(\mu)_{a, I} \cap \Adm(\mu)_{a', I'} = \emptyset$. If $\CF_{a, I} \cap \CF_{a', I'} = \CF_{a'', I''}$ for some $I''\subset\Pi_0$ and $a'' \in W^{I''}$, then $\Adm(\mu)_{a, I} \cap \Adm(\mu)_{a', I'} = \Adm(\mu)_{a'', I''}$.
\end{enumerate}
\end{proposition}
\begin{proof}
\begin{enumerate}
\item The ``if" direction follows from Theorem \ref{thm:main1}. Suppose $\Adm(\mu)_{a, I} = \Adm(\mu)_{a', I'}$. Note that the set of maximal elements of $\Adm(\mu)_{a, I}$ is $\{t^{\mu'} \mid \mu' \in aW_I(\mu)\}$. Then $a W_I(\mu) = a' W_{I'}(\mu)$. By Theorem \ref{thm:Maxwell} (2), we have $\CF_{a, I} = \CF_{a', I'}$ as desired.
    
\item The statements follow from Theorem \ref{thm:main1}. 
\end{enumerate}
\end{proof}

\begin{definition}
Let $\CF$ be a face of the polytope $\CP_{\mu}$. Define $\Adm(\mu)_{\CF} = \Adm(\mu)_{a,I}$ for any $I\subset\Pi_0$ and $a\in W^I$ with $\CF = \CF_{a,I}$. This is well-defined by Proposition \ref{prop:face-corr} (1). 
    
\end{definition}

We now define the notions of interior, boundary, and the center of a face $\Adm(\mu)_{\CF}$, which play important roles in analyzing the face map.

Let $\CF$ be a face of the polytope $\CP_{\mu}$. Define
\begin{align*}
\Adm(\mu)_{\CF}^{\circ} = \Adm(\mu)_{\CF} - \bigcup_{\CF_1 \subsetneq \CF} \Adm(\mu)_{\CF_1}.
\end{align*}
Here, $\CF_1$ runs over all faces of $\CP_{\mu}$ that are properly contained in $\CF$. We call $\Adm(\mu)_{\CF}^{\circ}$ the \emph{interior} of the face $\Adm(\mu)_{\CF}$. This is the analogue of the topological relative interior of the face $\CF$. We call $\Adm(\mu)_{\CF}-\Adm(\mu)_{\CF}^{\circ}$ the \emph{boundary} of $\Adm(\mu)_{\CF}$. %Denote $\Adm(\mu)^{\circ} = \Adm(\mu)_{\CP_{\mu}}^{\circ}$ for simplicity.

Let $\CF$ be a face of $\CP_{\mu}$. Choose $I\subset\Pi_0$ and $a \in W^I$ such that $\CF = \CF_{a, I}$. Let $c_{a, I}$ be the length-zero element in $\Adm^{\tW_{a, I}}(a(\mu))$. Then $c_{a, I}$ is the unique minimal element in $\Adm^{\tW_{a, I}}(a(\mu))$ with respect to the Bruhat order $\le_{a, I}$ of $\tW_{a, I}$. By \cite[Corollary 3.4]{Dyer90}, it is also the unique minimal element in $\Adm(\mu)_{a, I}$ with respect to the Bruhat order $\le$ of $\tW$. In particular, $c_{a, I}$ is independent of the choice of the pair $(I, a)$. We denote $c_{\CF} = c_{a, I}$. 

By Theorem \ref{thm:main1}, $c_{\CF}$ does not lie in $\Adm(\mu)_{\CF_1}$ for any smaller face $\CF_1 \subsetneq \CF$. In other words, we have
\begin{align}\label{eq:cF}
    c_{\CF} \in \Adm(\mu)_{\CF}^{\circ}.
\end{align}
We call $c_{\CF}$ the \emph{center} of the face $\Adm(\mu)_{\CF}$. This is the analogue of the barycenter of the face $\CF$ of $\CP_{\mu}$.

For minimal faces $\CF = \{\mu'\}$, $\mu' \in W_0(\mu)$, we have $c_{\CF} = t^{\mu'}$. If $\CF=\CP_{\mu}$ is the maximal face, then $c_{\CF} = \t_{\mu}$, the length-zero element in $\Adm(\mu)$. Figure 5 is an example in the case of type $A_2$ and $\mu = \omega_1^{\vee} + \omega_2^{\vee}$. Figure 6 is an example in the case of type $C_2$ and $\mu = \omega_2^{\vee}$. In both cases, the boundary consists of the alcoves shaded in light or dark gray. The centers of the ``1-dimensional" faces are shaded in dark gray.

%\newpage
\begin{figure}[htbp]
\centering
\begin{minipage}{0.45\textwidth}
\centering
\begin{tikzpicture}[scale=0.7]

\draw[fill=lightgray](0,-2)--(-0.5773502691896257,-1)--(0.5773502691896257,-1)--cycle;
\draw[fill=lightgray](0,-2)--(-0.5773502691896257,-1)--
(-1.7320508075688774,-1)--
(-1.1547005383792515,0)--
(-0.5773502691896257,-1)--
(0.5773502691896257,-1)--
(1.7320508075688774,-1)--
(1.1547005383792515,-2)--
cycle;
\draw[fill=lightgray]
(-1.1547005383792515,0)--
(-1.7320508075688774,1)--
(-1.1547005383792515,2)--
(0,2)--
(0.5773502691896257,3)--
(1.1547005383792515,2)--
(0,2)--
(-0.5773502691896257,1)--
(-1.1547005383792515,0);
\draw[fill=lightgray]
(1.1547005383792515,2)--
(2.309401076758503,2)--
(2.8867513459481287,1)--
(2.309401076758503,0)--
(2.8867513459481287,-1)--
(1.7320508075688774,-1)--
(2.309401076758503,0)--cycle;
\draw[fill=darkgray]
(-0.5773502691896257,1)--
(-1.1547005383792515,2)--
(0,2)--cycle;
\draw[fill=darkgray]
(1.1547005383792515,2)--
(2.309401076758503,2)--
(1.7320508075688774,1)--cycle;
\draw[fill=darkgray]
(0.5773502691896257,-1)--
(1.7320508075688774,-1)--
(1.1547005383792515,-2)--cycle;
\draw[fill=darkgray]
(-1.1547005383792515,0)--
(-1.7320508075688774,1)--
(-0.5773502691896257,1)--cycle;
\draw[fill=darkgray]
(-0.5773502691896257,-1)--
(0.5773502691896257,-1)--
(0,-2)--cycle;
\draw[fill=darkgray]
(1.7320508075688774,1)--
(2.8867513459481287,1)--
(2.309401076758503,0)--cycle;

\begin{scope}
\clip (-3.5, -3.5) rectangle (4.5, 4.5);
\foreach \i in {-4.618802153517006, -3.464101615137755, -2.309401076758503, -1.1547005383792515, 0.0, 1.1547005383792515, 2.309401076758503, 3.464101615137755, 4.618802153517006} {
\draw[black, thin] (\i,0) -- (\i+ 2.309401076758503, 4);
\draw[black, thin] (\i,0) -- (\i-2.309401076758503,-4);
\draw[black, thin] (\i,0) -- (\i-2.309401076758503,4);
\draw[black, thin] (\i,0) -- (\i+2.309401076758503,-4);
}
\foreach \k in {-4,-3,...,4} {
\draw[gray, thin] (-10, \k) -- (10, \k);
}
\end{scope}
\draw[line width=1mm] 
(-1.1547005383792515,0)--
(-1.7320508075688774,1)--
(-1.1547005383792515,2)--
(0,2)--
(0.5773502691896257,3)--
(1.1547005383792515,2)--
(2.309401076758503,2)--
(2.8867513459481287,1)--
(2.309401076758503,0)--
(2.8867513459481287,-1)--
(1.7320508075688774,-1)--
(1.1547005383792515,-2)--
(0,-2)--
(-0.5773502691896257,-1)--
(-1.7320508075688774,-1)
--cycle;
\draw[fill, black] (0,0) circle [radius=.11cm];
\end{tikzpicture}
\caption{}
\end{minipage}
\hfill
\begin{minipage}{0.45\textwidth}
\centering
\begin{tikzpicture}[scale=0.4]
\draw[fill=lightgray](0,0)--(-2,2)--(0,4)--(2,2)--(0,2)--cycle;

\draw[fill=lightgray](2,2)--(4,4)--(4,0)-- cycle;
\draw[fill=lightgray]
(0,0)--(-2,-2)--(4,-2)--(4,0)--(2,-2)--cycle;

\draw[fill=darkgray]
(0,0)--(-2,2)--(0,2)--cycle;
\draw[fill=darkgray]
(0,2)--(0,4)--(2,2)--cycle;
\draw[fill=darkgray]
(2,2)--(4,2)--(4,0)--cycle;
\draw[fill=darkgray]
(0,0)--(2,-2)--(0,-2)--cycle;

\draw[line width=1mm] (-2,2)--(0,4)--(2,2)--(4,4)--(4,-2)--(-2,-2)--(0,0)-- cycle;
\draw[fill, black] (0,0) circle [radius=.23cm];

\begin{scope}
\clip (-6.5, -6.5) rectangle (8, 7.5);
\foreach \i in {-8,-6,...,8} {
\draw[black, thin] (-10,\i) -- (10,\i);
\draw[black, thin] (\i,-10) -- (\i,10);
}
\foreach \k in {-4,-3,...,4} {
\draw[black, thin] (-10, {-10 + 4*\k}) -- (10, {10 + 4*\k});
\draw[black, thin] (-10, {10 - 4*\k}) -- (10, {-10 - 4*\k});
}
\end{scope}

\end{tikzpicture}
\caption{}
\end{minipage}
\end{figure}

Pappas--Rapoport \cite[(4.1.4)]{PR26} defined the \emph{face map}
\begin{align*}
\nabla:\Adm(\mu)\longrightarrow \mathscr{F}(\CP_\mu)
\end{align*}
as follows. For $w \in \Adm(\mu)$, define $\nabla(w)$ as the smallest face of $\CP_{\mu}$ containing $\L(w) = \{\mu' \in W_0(\mu) \mid w \le t^{\mu'}\}$.

It is clear that the face map $\nabla$ is order reversing, that is, if $w \le w'$ in $\Adm(\mu)$, then $\nabla(w) \supset \nabla(w')$. For $\mu' \in W_0(\mu)$, $\nabla(t^{\mu'})$ is the zero-dimensional face $\{\mu'\}$. If $\t_{\mu}$ is the unique length-zero element in $\Adm(\mu)$, then $\nabla(\t_{\mu})$ is the maximal face $ \CP_{\mu}$.

Pappas--Rapoport proved that the face map $\nabla$ coincides with the map induced from the divisor map and had conjectured that the face map $\nabla$ is surjective.

Our second main result gives an explicit description of the face map $\nabla$ (the parahoric version will be given in \S\ref{sec:para}).
\begin{theorem}\label{thm:main2}
The following are true.
\begin{enumerate}
\item For any face $\CF$ of $\CP_{\mu}$, the fiber $\nabla^{-1}(\CF)$ of the face map is equal to
$\Adm(\mu)_{\CF}^{\circ}$, the interior of the face $\Adm(\mu)_{\CF}$.

\item There is a decomposition
\begin{align*}
    \Adm(\mu) = \bigsqcup_{\CF} \Adm(\mu)_{\CF}^{\circ}.
\end{align*}
Here, $\CF$ runs over all faces of $\CP_{\mu}$. Moreover, for $w \in \Adm(\mu)$, $\nabla(w)$ is the unique face $\CF$ such that $w \in \Adm(\mu)_{\CF}^{\circ}$.

\item The face map $\nabla$ is surjective.
\end{enumerate}
\end{theorem}

\begin{proof}
\begin{enumerate}
\item By definition, for a face $\CF$ of $\CP_{\mu}$, the fiber $\nabla^{-1}(\CF)$ is equal to the set of $w\in\Adm(\mu)$ such that $\L(w)$ is contained in $\CF$ but is not contained in any smaller face. Then the statement follows from Theorem \ref{thm:main1} and the definition of $\Adm(\mu)_{\CF}^{\circ}$.

\item Note that $\Adm(\mu) = \bigsqcup_{\CF} \nabla^{-1}(\CF)$, where $\CF$ runs over all faces of $\CP_{\mu}$. Then the statement follows from (1).

\item By (1), it suffices to prove that $\Adm(\mu)_{\CF}^{\circ}$ is non-empty. This follows from the fact that the center element $c_{\CF}$ lies in $\Adm(\mu)_{\CF}^{\circ}$ (cf. (\ref{eq:cF})). 
\end{enumerate}
\end{proof}

Note that the face map $\nabla$ is not injective in general, since the interior $\Adm(\mu)_{\CF}^{\circ}$ may be large, see Figures 5, 6.

\begin{remark}
The decomposition in Theorem \ref{thm:main2} (2) is closely related to the Hodge--Newton decomposition 
\begin{align*}
    \Adm(\mu, b) = \bigsqcup_{P' = M' N' \in \mathscr{P}^{\s}} \Adm^{M'}(z_{P'}(\mu),b_{P'})
\end{align*}
in \cite[Theorem 4.11 (1)]{GHN2}. Note also that the interior $\Adm(\mu)^{\circ}$ is closely related to the Hodge--Newton indecomposable Newton strata. 
\end{remark}

\section{Proof of Theorem \ref{thm:main1}}\label{sec:proof}

We need some preparations. Recall that $\Phi_{a, I}$ is the subroot system of $\Phi$ spanned by $\{a(\a) \mid \a\in I\}$.

\begin{lemma}\label{lem:HH}
\begin{enumerate}
\item For any $w \in \Adm(\mu)$, $z \in W_0$, and $e\in \mathfrak{a}$, we have $w(e) - e = z(\mu - \l)$ for some $\l \in \BR_{\ge0}\Phi^{\vee,+}$. Here, $\Phi^{\vee,+}$ is the set of positive coroots.
\item Let $I\subset\Pi_0$ and $a \in W^I$. For any $w \in \Adm(\mu)_{a, I}$, $u \in W_I$, and $e\in \mathfrak{a}$, we have $w(e) - e = au(\mu -  \l)$ for some $\l \in \BR_{\ge0}\Phi_I^{\vee,+}$. Here, $\Phi_I^{\vee,+} = \Phi_I^{\vee} \cap \Phi^{\vee,+}$.
\end{enumerate}
\end{lemma}
\begin{proof}
(1) follows from Theorem \ref{thm:HH} and the definition of $O(t^{z(\mu)},z)$. We now prove (2). Note that $\Adm(\mu)_{a, I} = \Adm^{\tW_{a, I}}(a(\mu))$. Applying (1) to $\Adm^{\tW_{a, I}}(a(\mu))$, we get $w(e) -e = a u a^{-1}(a(\mu) - \l')$ for some $\l' \in \BR_{\ge0}\Phi_{a, I}^{\vee, +}$. Here, $\Phi_{a, I}^{\vee, +} = \Phi_{a, I}^{\vee} \cap \Phi^{\vee,+}$. Since $a \in W^I$, we have $a^{-1}(\Phi_{a, I}^{\vee,+}) = \Phi_I^{\vee,+}$. Then $w(e) - e = a u(\mu-\l)$ for some $\l \in \BR_{\ge0}\Phi_I^{\vee,+}$.
\end{proof}

Recall that the set of affine roots is $\tilde\Phi = \Phi \times \BZ$. Elements in $\tW$ act on $\tilde\Phi$ via $t^{\l}z(\g,k) = (z(\g),k-\<\l,z(\g)\>)$. By convention, the set of positive affine roots is $\tilde\Phi^+ = (\Phi^+\times \BZ_{\ge0}) \sqcup (\Phi^-\times \BZ_{\ge1})$. For $w \in \tW$ and $\tg\in\tilde\Phi^+$, we have $w < s_{\tg} w \iff w^{-1}(\tg) \in \tilde\Phi^+$.

\begin{lemma}\label{lem:key}
Let $I\subset\Pi_0$ and $a \in W^I$. Let $w \in \Adm(\mu)_{a, I}$ and $\tilde\g  \in \tilde{\Phi}^+$. Suppose $w < s_{\tilde\g} w$ and $s_{\tilde\g} w \in \Adm(\mu)$. Then $\tg \in \Phi_{a, I} \times \BZ$.
\end{lemma}
\begin{proof}
Write $\tilde\g = (\g,k)$. Take an arbitrary $e\in\mathfrak{a}$. Since $(\g,k)\in \tilde{\Phi}^+$, we have $k + \<e,\g\> > 0$. As $w \in \Adm(\mu)_{a, I} = \Adm^{\tW_{a, I}}(a(\mu))$, by Lemma \ref{lem:HH} (2), we have
\begin{align}\label{eq:w(e)}
w(e) = e + a(\mu - \l)    
\end{align}
for some $\l \in \BR_{\ge0}\Phi_{I}^{\vee,+}$. Since $w < s_{\tilde\g} w$, we have $k + \<e + a(\mu-\l), \g\> > 0$.

Using (\ref{eq:w(e)}), direct computation shows that
$$s_{\tilde\g}w(e) - e = s_\g a\left(\mu- \l + (k+\<e,\g\>) a^{-1}(\g^{\vee})\right),$$
and 
$$s_{\tilde\g}w(e) - e = a\left(\mu - \l - (k+\<e + a(\mu-\l),\g\>) a^{-1}(\g^{\vee})\right).$$

Suppose $\g \notin  \Phi_{a, I}$. Then $a^{-1}(\g^{\vee}) \notin \Phi_I^{\vee}$. Note that $\l \in \BR_{\ge0}\Phi_I^{\vee,+}$, $k+\<e,\g\> >0$ and $k+\<e + a(\mu-\l),\g\> >0 $. Then 
$\l - (k+\<e,\g\>) a^{-1}(\g^{\vee})$ or $\l + (k+\<e + a(\mu-\l),\g\> ) a^{-1}(\g^{\vee})$ does not lie in $\BR_{\ge0}\Phi^{\vee,+}$. By Lemma \ref{lem:HH} (1), we have $s_{\tilde\g} w \notin \Adm(\mu)$. A contradiction. Hence $\g \in \Phi_{a, I}$.
\end{proof}
\begin{remark}
A variant of this lemma is proved in \cite[Lemma 4.12]{GHN2} by G\"ortz--He--Nie using the technique of $(I, a)$-alcove elements.
\end{remark}

\begin{corollary}\label{cor:Bruhat-closed}
\begin{enumerate}
\item Let $I \subset \Pi_0$ and $a \in W^I$. Then $\Adm(\mu)_{a, I}$ is upwardly closed under Bruhat order, that is, for $w_1,w_2\in\Adm(\mu)$, if $w_1<w_2$ and $w_1\in\Adm(\mu)_{a, I}$, then $w_2 \in \Adm(\mu)_{a, I}$.
\item The restriction of the Bruhat order $\le_{a, I}$ of $\tW_{a, I}$ to $\Adm^{\tW_{a, I}}(a(\mu))$ coincides with the restriction of the Bruhat order $\le$ of $\tW$ to $\Adm(\mu)_{a, I}$.
\end{enumerate}
\end{corollary}
\begin{proof}
\begin{enumerate}
\item Since $w_1 < w_2$, by definition of Bruhat order, there is a chain $w_1 < s_{\tilde\g_1} w_1 <s_{\tilde\g_2} s_{\tilde\g_1} w_1 <\cdots < s_{\tilde\g_r} \cdots s_{\tilde\g_1}w_1 = w_2$. Using Lemma \ref{lem:key} repeatedly, we conclude that all $\tilde\g_i \in \Phi_{a, I} \times \BZ$ and hence $w_2 \in  W_{a, I, \aff}w_1 \cap \Adm(\mu) = \Adm(\mu)_{a, I}$ by Lemma \ref{lem:def2}. 

\item Let $w,w' \in \Adm(\mu)_{a, I}$. If $w \le_{a, I} w'$, then \cite[Corollary 3.4]{Dyer90} implies $w \le w'$. 

We now prove the converse. Suppose $w < w'$. Then there is a chain $w < s_{\tilde\g_1} w <s_{\tilde\g_2} s_{\tilde\g_1} w <\cdots < s_{\tilde\g_r} \cdots s_{\tilde\g_1}w = w'$. By (1), all terms in the chain lie in $\Adm(\mu)_{a, I} = \Adm^{\tW_{a, I}}(a(\mu))$. Again by \cite[Corollary 3.4]{Dyer90}, we have $w <_{a, I} s_{\tilde\g_1} w <_{a, I} s_{\tilde\g_2} s_{\tilde\g_1} w <_{a, I} \cdots <_{a, I} s_{\tilde\g_r} \cdots s_{\tilde\g_1}w = w'$. This completes the proof.
\end{enumerate}
\end{proof}

We need one more lemma.
\begin{lemma}\label{lem:a1a2}
Let $I\subsetneq\Pi_0$. Let $\a_1,\a_2\in\Phi^+-\Phi_I^+$. Suppose $s_{\a_1}s_{\a_2} \in W_I$. Then $\a_1 = \a_2$. 
\end{lemma}
\begin{proof} 
For $z \in W_0$, denote by $\supp(z)$ the support of $z$, that is, the set of $\a\in\Pi_0$ such that $s_\a$ lies in some (or any) reduced expression of $z$.

Set $u = s_{\a_1}s_{\a_2} \in W_I$. Because $\a_1 \in \Phi^+ - \Phi_I^+$, its support must contain at least one simple root not in $I$. Let $i \in \supp(s_{\a_1}) - I$. Let $\omega_i^{\vee}$ be the fundamental coweight corresponding to $i$. Since $u \in W_I$ and $i \notin I$, we have $u(\omega_i^{\vee})=\omega_i^{\vee}$. On the other hand, expanding the action of $u$ gives $u(\omega_i^{\vee}) = s_{\a_1} s_{\a_2} (\omega_i^{\vee}) =\omega_i^{\vee} -\<\omega_i^{\vee},\a_1\>\a_1^{\vee} - \<\omega_i^{\vee},\a_2\>\a_2^{\vee}+ \<\omega_i^{\vee},\a_2\>\<\a_2^{\vee},\a_1\>\a_1^{\vee}$. If $\a_1$ and $\a_2$ are linearly independent, then both $\<\omega_i^{\vee},\a_1\>$ and $\<\omega_i^{\vee},\a_2\>$ vanish. This contradicts the assumption that $i\in \supp(s_{\a_1})$. Hence, $\a_1=\a_2$.
\end{proof}
We now start to prove Theorem \ref{thm:main1}.

\begin{proof}[Proof of Theorem \ref{thm:main1}]Let $I\subset \Pi_0$ and $a \in W^I$. Note that $\L(w)\subset \CF_{a, I}$ if and only if $\L(w)\subset aW_I(\mu)$.

We first prove that $\Adm(\mu)_{a, I} \subset \{w \in \Adm(\mu) \mid \L(w) \subset a W_I(\mu)\}$. Let $w \in \Adm(\mu)_{a, I}$. Suppose $w < t^{\mu'}$ for some $\mu' \in W_0(\mu)$. By Corollary \ref{cor:Bruhat-closed} (1), we get $t^{\mu'} \in \Adm(\mu)_{a, I}$. Then $\mu' \in aW_I(\mu)$.

Next, we prove that $\{w \in \Adm(\mu) \mid \L(w) \subset a W_I(\mu)\} \subset \Adm(\mu)_{a, I}$. Let $w \in \Adm(\mu) - \Adm(\mu)_{a, I}$. It suffices to prove that $w \le t^{\mu'}$ for some $\mu' \notin aW_I(\mu)$. Proceed by downward induction on $\ell(w)$. If $\ell(w) = \<\mu,2\rho\>$, the statement is clear. Now suppose $\ell(w) < \<\mu,2\rho\>$. 

If there exists $w' \in \Adm(\mu)- \Adm(\mu)_{a, I}$ with $w' > w$, then the statement follows from applying the induction hypothesis to $w'$. Now suppose for all $w' \in \Adm(\mu)$ such that $w' > w$, we have $w' \in \Adm(\mu)_{a, I}$. By \cite[Proposition 3.14 (2)]{HSY26}, one can find two affine roots $\tilde\g_1 = (\g_1,k_1), \tilde\g_2 = (\g_2, k_2) \in \tilde\Phi^+$ with $\tg_1\ne \tg_2$, $w < s_{\tilde\g_1}w$, $w < s_{\tilde\g_2}w$ and $s_{\tilde\g_1}w,s_{\tilde\g_2}w \in \Adm(\mu)_{a, I}$. 

Since $w \in \Adm(\mu) - \Adm(\mu)_{a, I}$ and $s_{\tilde\g_1}w, s_{\tilde\g_2}w \in \Adm(\mu)_{a, I}$, by Lemma \ref{lem:def2}, we have $\g_1,\g_2 \notin \Phi_{a, I}$ and $s_{\tg_1}s_{\tg_2} \in W_{a, I, \aff}$. By Lemma \ref{lem:a1a2}, we have $\g_1 = c \g_2$ with $c = \pm1$. Note that $s_{\tg_1}s_{\tg_2} = s_{\g_1}s_{\g_2}t^{k_1s_{\g_2}(\g_1^{\vee}) + k_2 \g_2^{\vee}} = t^{k_1s_{\g_2}(\g_1^{\vee}) + k_2 \g_2^{\vee}}=t^{(- k_1 c + k_2) \g_2^{\vee}} $. It follows that $(- k_1 c + k_2) \g_2^{\vee} \in \BZ\Phi_{a, I}^{\vee}$. But $\g_2\notin \Phi_{a, I}$. Hence $k_2 = k_1 c$ and $\tg_1 = c\tg_2$. This contradicts the assumption that $\tg_1,\tg_2 \in \tilde\Phi^+$ and $\tg_1 \ne \tg_2$. This completes the proof.
\end{proof}

\section{Parahoric Level}\label{sec:para}
Let $K$ be a subset of affine simple roots. Denote by $W_K$ the subgroup of $\tW$ generated by $\{s_\a \mid \a \in K\}$. We say $K$ is \emph{spherical} if $W_K$ is finite.

In the rest of this paper, we fix a spherical subset $K$.

Let $\pr$ be the natural projection $\tilde\Phi \to \Phi$ sending $(\a, m)$ to $\a$. Let $\Phi_{\pr(K)}$ be the subroot system of $\Phi$ spanned by $\{\pr(\a)\mid \a\in K\}$ and $\Phi_{\pr(K)}^{\vee}$ its set of coroots. Let $V_K = (X_*/ \BZ\Phi_{\pr(K)}^{\vee}) \otimes \BR$. The natural projection $X_* \to X_*/ \BZ\Phi_{\pr(K)}^{\vee}$ induces
\begin{align*}
    \pi_K: V \to V_K.
\end{align*}
Let $W_{\pr(K)}$ be the subgroup of $W_0$ generated by $\{s_{\pr(\ta)} \mid \ta \in K\}$. It is clear that $\pi_K(W_{\pr(K)}(v)) = \pi_K(v)$ for any $v \in V$. 

\begin{example}
Suppose $\Phi$ is irreducible and $K = \{\ta_0\}$; we have $W_K = \{1,s_0 = t^{\th^{\vee}}s_\th\}$ and  $W_{\pr(K)} = \{1, s_{\th}\}$. Here, $\th$ is the highest root.
\end{example}

Let $\CP_{K, \mu} = \pi_K(\CP_{\mu})$ be the image of the coweight polytope under $\pi_K$. Then $\CP_{K, \mu}$ is the convex hull of $\{\pi_K(\mu') \mid  \mu' \in W_0(\mu)\}$, and hence it is a polytope. Let $\mathscr{F}(\CP_{K, \mu})$ be the poset of nonempty (closed) faces of $\CP_{K, \mu}$. The following result shows that the faces of $\CP_{K, \mu}$ are exactly the images of those faces of $\CP_{\mu}$ that are stable under $W_{\pr(K)}$.

\begin{proposition}\label{prop:faceK}
There is a bijection
\begin{align*}
\{\CF \in  \mathscr{F}(\CP_{\mu}) \mid W_{\pr(K)}(\CF) = \CF\}    &\xlongrightarrow{\sim} \mathscr{F}(\CP_{K, \mu})\\ 
\CF &\longmapsto \pi_K(\CF).
\end{align*}
\end{proposition}
\begin{proof}
Let $\CF'$ be a face of $\CP_{K, \mu}$. By \cite[Lemma 7.10]{Zie95}, the preimage $\left(\pi_K|_{\CP_{\mu}}\right)^{-1}(\CF')$ is a face of $\CP_{\mu}$. Then $\CF' = \pi_K\left(\left(\pi_K|_{\CP_{\mu}}\right)^{-1}(\CF')\right)$. Since $\pi_K(W_{\pr(K)}(v)) = \pi_K(v)$ for any $v \in V$, the preimage $\left(\pi_K|_{\CP_{\mu}}\right)^{-1}(\CF')$ is stable under $W_{\pr(K)}$.

Conversely, let $\CF$ be a face of $\CP_{\mu}$ that is stable under $W_{\pr(K)}$. We shall prove that $\pi_K(\CF)$ is a face of $\CP_{K, \mu}$. There is a linear function $\g$ and $M \in \BR$ such that $\g(v) \le M$ for all $v \in \CP_{\mu}$ and $\CF = \{v \in \CP_{\mu} \mid \g(v) = M\}$. Since both $\CP_{\mu}$ and $\CF$ are $W_{\pr(K)}$-stable, we may assume that $\g$ is $W_{\pr(K)}$-invariant. Then $\g$ descends to a linear function $\bar{\g}$ on $V_K$ such that $\bar{\g}(\pi_K(v)) = \g(v)$ for all $v \in V$. It follows that $\pi_K(\CF)$ is the face of $\CP_{K, \mu} = \pi_K(\CP_{\mu})$ determined by the linear function $\bar{\g}$. Moreover, it is clear that $\left(\pi_K|_{\CP_{\mu}}\right)^{-1}(\pi_K(\CF)) = \CF$. This completes the proof.
\end{proof}

By abuse of notation, we also denote by $\pi_K$ the projection $\tW \to W_K \backslash \tW / W_K$. For any $w \in \tW$, denote by $\<w\> = \pi_K(w)$ the image of $w$ in the double quotient $W_K \backslash \tW / W_K$. By abuse of notation, we also use $\le$ to denote the induced Bruhat order on $W_K \backslash \tW / W_K$. \footnote{It is convenient to use both $\pi_K(w)$ and $\<w\>$.}

Define $\Adm(\mu)_K = \pi_K(\Adm(\mu))$. Then
\begin{align*}
    \Adm(\mu)_K = \{\<w\> \in W_K \backslash \tW / W_K \mid \<w\> \le \<t^{\mu'}\> \text{ for some }\mu'\in W_0(\mu)\}.
\end{align*}  
Note that the set of maximal elements in $\Adm(\mu)_K$ is $\{\<t^{\mu'}\> \mid \mu' \in W_0(\mu)\}$ (see \cite[Theorem 1.4]{HY24}).

Recall that for any $w \in \Adm(\mu)$, $\L(w) = \{\mu' \mid \mu' \in W_0(\mu), w \le t^{\mu'}\}$ and $\nabla(w)$ is the smallest face of $\CP_{\mu}$ containing $\L(w)$. For any $\<w\> \in \Adm(\mu)_K$, define
\begin{align*}
\L_K(\<w\>) = \{\pi_K(\mu') \mid \mu' \in W_0(\mu), \<w\> \le \<t^{\mu'}\>\}.
\end{align*}
Define $\nabla_K(\<w\>)$ as the smallest face of $\CP_{K, \mu}$ containing $\L_K(\<w\>)$.

Let $\CF'$ be a face of $\CP_{K, \mu}$. Let $\CF$ be the preimage of $\CF'$ in $\CP_{\mu}$. Define $\Adm(\mu)_{K, \CF'} = \pi_K\left(\Adm(\mu)_{\CF}\right)$. As in the Iwahori case, define $
\Adm(\mu)_{K, \CF'}^{\circ} = \Adm(\mu)_{K, \CF'} - \bigcup_{\CF_1'}\Adm(\mu)_{K, \CF_1'}$, where $\CF_1'$ runs over all faces of $\CP_{K, \mu}$ that are properly contained in $\CF'$. We call $\Adm(\mu)_{K, \CF'}$ a \emph{face} of $\Adm(\mu)_K$ and $\Adm(\mu)_{K, \CF'}^{\circ}$ its \emph{interior}.

The main result of this section is the following.

\begin{theorem}\label{thm:main3}
Let $\CF'$ be a face of $\CP_{K, \mu}$ and let $\CF$ be the preimage of $\CF'$ in $\CP_{\mu}$. The following statements are true.
\begin{enumerate}

\item The preimage of $\Adm(\mu)_{K, \CF'}$ in $\Adm(\mu)$ is $\Adm(\mu)_{\CF}$.

\item We have
\begin{align*}
\Adm(\mu)_{K, \CF'} = \{ \<w\> \in \Adm(\mu)_K \mid \L_K(\<w\>) \subset \CF'\}.
\end{align*}

\item We have
\begin{align*}
\Adm(\mu)_{K, \CF'}^{\circ} = \pi_K(\Adm(\mu)_{\CF}^{\circ}).
\end{align*}

\item The fiber $(\nabla_K)^{-1}(\CF')$ is equal to $\Adm(\mu)_{K, \CF'}^{\circ}$.

\item The map $\nabla_K$ is surjective.

\item There is a decomposition
\begin{align*}
    \Adm(\mu)_K = \bigsqcup_{\CF''} \Adm(\mu)_{K, \CF''}^{\circ}.
\end{align*}
Here, $\CF''$ runs over all faces of $\CP_{K, \mu}$.
\end{enumerate}
\end{theorem}
Let $\CF'$ be a face of $\CP_{K, \mu}$ and let $\CF$ be the preimage of $\CF'$ in $\CP_{\mu}$. Let $c_{\CF}$ be the center of $\Adm(\mu)_{\CF}$ (see \S\ref{sec:main}). We have $\pi_K(c_{\CF}) \in \Adm(\mu)_{K, \CF'}^{\circ}$ by Theorem \ref{thm:main3} (3). We call $\pi_K(c_{\CF})$ the \emph{center} of the face $\Adm(\mu)_{K, \CF'}$. 

The proof of Theorem \ref{thm:main3} will be given in \S\ref{sec:proof2}.

Finally, we study the relationship between $\nabla$ and $\nabla_K$.

Note that for a face $\CF$ of $\CP_{\mu}$, the image $\pi_K(\CF)$ may not be a face of $\CP_{K, \mu}$. Define $\tilde{\pi}_K(\CF)$ as the smallest face of $\CP_{K, \mu}$ containing the image $\pi_K(\CF)$. 
\begin{proposition}\label{prop:comm}
The following diagram is commutative.
$$
\begin{tikzcd}
\Adm(\mu) \arrow[r, "\nabla"] \arrow[d, "\pi_K"] & \mathscr{F}(\CP_{\mu}) \arrow[d, "\tilde{\pi}_K"] \\
\Adm(\mu)_K \arrow[r, "\nabla_K"] & \mathscr{F}(\CP_{K, \mu}).
\end{tikzcd}
$$
In other words, for any $w \in \Adm(\mu)$, $\nabla_K(\<w\>)$ is the smallest face of $\CP_{K, \mu}$ containing $\pi_K(\nabla(w))$. Moreover, we have $\nabla_K(\<w\>) = \pi_K(\nabla(w_{\min}))$, where $w_{\min}$ is the unique minimal element
of the double coset $W_KwW_K$.
\end{proposition}

The proof of Proposition \ref{prop:comm} will be given in \S\ref{sec:proof2}.

\section{Proof of Theorem \ref{thm:main3} and Proposition \ref{prop:comm}}\label{sec:proof2}
In order to prove Theorem \ref{thm:main3}, we need the notions of $J_{\mu}$-complete representatives of faces of the coweight polytope $\CP_{\mu}$ and the minimal representatives of double cosets in $\tW$.

\subsection{The $J_{\mu}$-complete representatives of faces}

For a face $\CF$ of $\CP_{\mu}$, there may be more than one pair $(I, a)$ such that $\CF = \CF_{a, I}$. There is a canonical "maximal" choice which we now introduce.

Let $I\subset\Pi_0$. Denote by $I_0$ the union of the connected components of $I$ that are not contained in $J_{\mu}=\{\a\in \Pi_0 \mid \<\mu,\a\> = 0\}$. Suppose that there exists some $\a \in J_{\mu}$ that is not connected to $I_0$, i.e., such that $\{\a\}$ is a connected component of $I_0 \sqcup \{\a\}$. Then it is clear that $W_{I}(\mu) = W_{I\cup\{\a\}}(\mu)$. Including all such $\a$, we call the set 
\begin{align*}
    \tilde{I} = I_0 \sqcup \{\a \in J_{\mu} \mid \a \text{ is not connected to }I_0\}
\end{align*}
the \emph{$J_{\mu}$-completion} of $I$. It is clear that $I \subset \tilde I$. If $I = \tilde {I}$, we say $I$ is \emph{$J_{\mu}$-complete}.

\begin{example}
In type $A_5$, assume $J_{\mu} = \{\a_3, \a_4, \a_5\}$. We see that the $J_{\mu}$-completion of $\{\a_1, \a_2, \a_3\}$ is $\{\a_1,\a_2,\a_3,\a_5\}$, that the $J_{\mu}$-completion of $\{\a_1, \a_2\}$ is $\{\a_1,\a_2,\a_4,\a_5\}$ and that $\{\a_1, \a_2, \a_3,\a_4\}$ is $J_{\mu}$-complete.
\end{example}

\begin{theorem}[{\cite[Section 2 and 5]{MAX89}}]\label{thm:complete}
Let $\CF$ be a face of $\CP_{\mu}$. Then
\begin{enumerate}
\item there is a unique pair $(\tilde{I}, \tilde{a})$ where $\tilde{I}$ is a $J_{\mu}$-complete subset of $\Pi_0$ and $\tilde{a} \in W^{\tilde{I}}$ such that $\CF = \CF_{\tilde{a}, \tilde{I}}$;
\item for any $I \subset \Pi_0$ and $a \in W^I$ such that $\CF = \CF_{a, I}$, we have $I \subset \tilde{I}$.
\end{enumerate}
\end{theorem}
We call the pair $(\tilde{I}, \tilde{a})$ in Theorem \ref{thm:complete} (1) the \emph{$J_{\mu}$-complete representative} of the face $\CF$. For $I$ as in Theorem \ref{thm:complete} (2), we call $\tilde I$ its \emph{$J_{\mu}$-completion}. The $J_{\mu}$-complete representative is ``maximal" among all representatives.

\begin{lemma}\label{lem:complete}
Let $\CF$ be a face of $\CP_{\mu}$. Let $(I, a)$ be the $J_{\mu}$-complete representative of $\CF$. Then the following are equivalent.
\begin{enumerate}
\item $\CF$ is stable under $W_{\pr(K)}$.

%\item $W_{\pr(K)} a W_I = a W_I$.

\item $W_{\pr(K)} \subset W_{a, I}$.

\item $\Phi_{\pr(K)} \subset \Phi_{a,I}$.

\item $W_K \subset W_{a, I, \aff}$ (see \S\ref{sec:adm} for the notation).
\end{enumerate}
\end{lemma}

\begin{proof}

The implications $(2) \iff (3) \iff (4)$ follow from the definition. The implication $(2) \implies (1)$ follows from the simple fact that $W_{a, I} = a W_I a^{-1}$ stabilizes $\CF_{a, I}$. We now prove $(1) \implies (2)$. Suppose $\CF$ is stable under $W_{\pr(K)}$.
For any $x \in W_{\pr(K)}$, let $a' \in W^I$ be such that
$xa \in a'W_I$. Then
$\CF_{a,I} = x(\CF_{a,I}) = \CF_{a',I}$.
By Theorem \ref{thm:complete} (1), we have $a=a'$.
Thus, $aW_I=xaW_I$, so $a^{-1}xa\in W_I$.
Hence $x\in W_{a,I}$, as desired.
\end{proof}

\begin{remark}
The $J_{\mu}$-complete condition is essential for Lemma \ref{lem:complete} (1) $\implies$ (2). For example, let $I\subset \Pi_0$ and let $\tilde I$ be its $J_{\mu}$-completion. Suppose $I \subsetneq \tilde{I}$. Take $K = \tilde{I}$. Then the face $\CF_{1, I} = \CF_{1,\tilde{I}}$ is stable under $W_{\pr(K)}$. Note that $W_{\pr(K)} \subset W_{\tilde{I}}$ but $W_{\pr(K)} \not\subset W_{I}$. 
\end{remark}

\subsection{Minimal representatives of double cosets}

For any $w \in \tW$, there is a unique minimal element $w_{\min}$ in the double coset $W_K w W_K$ (see, for example, \cite[Proposition 2.7.5]{Car93}). Note that for any $w' \in \tW$, we have $\<w\> \le \<w'\>  \iff w_{\min} \le w'$. Let $w \in \Adm(\mu)$. We have
\begin{align*}\tag{6.1}\label{eq:wmin}
\L_K(\<w\>) = \pi_K(\L(w_{\min})).    
\end{align*}
Indeed, suppose $w_{\min} \le t^{\mu'}$ for some $\mu' \in W_0(\mu)$, then $\<w\> = \<w_{\min}\> \le \<t^{\mu'}\>$. Conversely, suppose $\<w\> \le \<t^{\mu'}\>$ for some $\mu' \in W_0(\mu)$, then $w_{\min} \le t^{\mu'}$.

The following lemma is the key to the proof of Theorem \ref{thm:main3}.

\begin{lemma}\label{lem:key2}
Let $\CF$ be a $W_{\pr(K)}$-stable face of $\CP_{\mu}$. Let $w \in \Adm(\mu)_{\CF}$. Let $w_{\min}$ be the unique minimal element in $W_K w W_K$. Then $w_{\min} \in \Adm(\mu)_{\CF}$. 
\end{lemma}

\begin{proof}
Let $(I,a)$ be the $J_\mu$-complete representative of $\CF$. By Lemma~\ref{lem:def2}, we have $\Adm(\mu)_{\CF} = \Adm(\mu)\cap W_{a,I,\mathrm{aff}}t^{a(\mu)}$. It is clear that $w_{\min}\in\Adm(\mu)$. It suffices to show $w_{\min} \in W_{a,I,\mathrm{aff}}t^{a(\mu)}$.

By Lemma \ref{lem:complete} (4), we have $W_K\subset W_{a,I,\mathrm{aff}}$.
Moreover, $t^{a(\mu)}$ normalizes $W_{a,I,\mathrm{aff}}$.
Thus, the coset $W_{a,I,\mathrm{aff}}t^{a(\mu)}$ is stable
under both left and right multiplication by $W_K$. Then, $w_{\min} \in W_KwW_K\subset
W_{a,I,\mathrm{aff}}t^{a(\mu)}$.
\end{proof}

\subsection{Proof of Theorem \ref{thm:main3}}

Let $\CF'$ be a face of $\CP_{K, \mu}$ and let $\CF$ be the preimage of $\CF'$ in $\CP_{\mu}$. 

We now prove Theorem \ref{thm:main3} (1). Let $w \in \Adm(\mu)$ such that $\pi_K(w) \in \Adm(\mu)_{K, \CF'}$. By definition of $\Adm(\mu)_{K, \CF'}$, $\pi_K(w) = \pi_K(w')$ for some $w' \in \Adm(\mu)_{\CF}$. Let $w_{\min}$ be the unique minimal element in $W_K w W_K = W_K w' W_K$. By Lemma \ref{lem:key2}, $w_{\min} \in \Adm(\mu)_{\CF}$. Since $w \ge w_{\min}$, by Corollary \ref{cor:Bruhat-closed} (1), we have $w \in \Adm(\mu)_{\CF}$, as desired.

We now prove Theorem \ref{thm:main3} (2). Let $\<w\> \in \Adm(\mu)_{K}$. Let $w_{\min}$ be the unique minimal element in the double coset $W_K w W_K$. We have
\begin{align*}
    \<w\> \in \Adm(\mu)_{K, \CF'}  &\iff \<w_{\min}\> \in \Adm(\mu)_{K, \CF'}\\  
    &\iff w_{\min} \in \Adm(\mu)_{\CF}\\
    &\iff \L(w_{\min}) \subset \CF\\
    &\iff \pi_K(\L(w_{\min})) \subset \CF'\\
    &\iff \L_K(\<w\>) \subset \CF'.
\end{align*}
Here, the second equivalence follows from Theorem \ref{thm:main3} (1), the third equivalence follows from Theorem \ref{thm:main1}, the fourth equivalence follows from Proposition \ref{prop:faceK}, and the fifth equivalence follows from (\ref{eq:wmin}).

We now prove Theorem \ref{thm:main3} (3). We first prove the $\supset$ direction. Let $w \in \Adm(\mu)_{\CF}^{\circ}$. Then $\pi_K(w) \in \Adm(\mu)_{K, \CF'}$. Suppose $\pi_K(w) \in \Adm(\mu)_{K, \CF'_1}$ for some smaller face $\CF'_1$ of $\CP_{K, \mu}$. Then, by (1), $w \in \Adm(\mu)_{\CF_1}$, where $\CF_1$ is the preimage of $\CF_1'$ in $\CP_{\mu}$. This contradicts the fact that $w \in \Adm(\mu)_{\CF}^{\circ}$. Thus, we conclude that $\pi_K(w) \in \Adm(\mu)_{K, \CF'}^{\circ}$.

Next, we prove the $\subset$ direction. Let $\<w\> \in \Adm(\mu)_{K, \CF'}^{\circ}$. Let $w_{\min}$ be the unique minimal element in the double coset $W_K w W_K$. Note that $\L(w_{\min})$ is $W_{\pr(K)}$-stable (indeed, for $\mu' \in \L(w_{\min})$ and $x \in W_{\pr(K)}$, it is easy to see that $t^{x(\mu')} \in W_K t^{\mu'} W_K$, then $w_{\min} \le t^{x(\mu')}$). Then, the face $\nabla(w_{\min})$ is also $W_{\pr(K)}$-stable. We claim that $w_{\min} \in \Adm(\mu)_{\CF}^{\circ}$. Then $\<w\> = \<w_{\min}\> \in \pi_K(\Adm(\mu)_{\CF}^{\circ})$ as desired. Suppose $w_{\min} \in \Adm(\mu)_{\CF_0}$ for some (not necessarily $W_{\pr(K)}$-stable) face $\CF_0 \subsetneq \CF$. Then
\begin{align*}\tag{6.2}\label{eq:6.2}
\nabla(w_{\min}) \subset \CF_0 \subsetneq \CF.    
\end{align*}
By (\ref{eq:wmin}), (\ref{eq:6.2}) and Proposition \ref{prop:faceK}, we have 
\begin{align*}
\L_K(\<w\>) = \pi_K(\L(w_{\min})) \subset \pi_K(\nabla(w_{\min})) \subsetneq \pi_K(\CF) = \CF',   
\end{align*}
that is, $\L_K(\<w\>)$ lies in the smaller face $\pi_K(\nabla(w_{\min})) \subsetneq \CF'$. By Theorem \ref{thm:main3} (2), this contradicts the assumption that 
$\<w\> \in \Adm(\mu)_{K, \CF'}^{\circ}$. This proves the claim.

The proofs of Theorem \ref{thm:main3} (4) and (6) are the same as those of Theorem \ref{thm:main2} (1) and (2). Theorem \ref{thm:main3} (5) follows from (3) and (4).

This completes the proof of Theorem \ref{thm:main3}.

\subsection{Proof of Proposition \ref{prop:comm}}

Let $w\in\Adm(\mu)$. For any face $\CF'$ of $\CP_{K,\mu}$,
let $\CF$ be its preimage in $\CP_\mu$.
By Theorem~\ref{thm:main3} (1), (2), the inclusion
$\nabla_K(\<w\>)\subset\CF'$ holds if and only if
$w\in\Adm(\mu)_\CF$.
By Theorem~\ref{thm:main1}, this is equivalent to
$\nabla(w)\subset\CF$, or equivalently,
$\pi_K(\nabla(w))\subset\CF'$.
Thus $\nabla_K(\<w\>)$ is the smallest face containing
$\pi_K(\nabla(w))$, proving the commutativity.

As shown in the proof of Theorem~\ref{thm:main3} (3),
the face $\nabla(w_{\min})$ is $W_{\pr(K)}$-stable.
Hence $\pi_K(\nabla(w_{\min}))$ is a face of $\CP_{K,\mu}$
by Proposition~\ref{prop:faceK}.
Applying the first assertion to $w_{\min}$ gives
$\nabla_K(\<w\>) = \nabla_K(\<w_{\min}\>) = \pi_K(\nabla(w_{\min}))$. This completes the proof of Proposition \ref{prop:comm}.


\begin{thebibliography}{99}

\bibitem{AGLR}
J.~Ansch\"utz, I.~Gleason, J.~Louren\c{c}o, and T.~Richarz,
\emph{On the $p$-adic theory of local models},
preprint, arXiv:2201.01234 (2022).

\bibitem{Car93}
R.~W.~Carter,
\emph{Finite groups of Lie type: Conjugacy classes and complex characters},
Wiley Classics Library, Wiley, 1993.

\bibitem{Dyer90}
M.~Dyer,
\emph{Reflection subgroups of Coxeter systems},
J. Algebra \textbf{135} (1990), no.~1, 57--73.

\bibitem{GHN2}
U.~G\"ortz, X.~He, and S.~Nie,
\emph{Fully Hodge--Newton decomposable Shimura varieties},
Peking Math. J. \textbf{2} (2019), no.~2, 99--154.

\bibitem{Ha01}
T.~J.~Haines,
\emph{The combinatorics of Bernstein functions},
Trans. Amer. Math. Soc. \textbf{353} (2001), no.~3, 1251--1278.



\bibitem{HH17}
T.~J.~Haines and X.~He,
\emph{Vertexwise criteria for admissibility of alcoves},
Amer. J. Math. \textbf{139} (2017), no.~3, 769--784.

\bibitem{HR23}
T.~J.~Haines and T.~Richarz,
\emph{Normality and Cohen--Macaulayness of parahoric local models},
J. Eur. Math. Soc. (JEMS) \textbf{25} (2023), no.~2, 703--729.

\bibitem{HSY26}
X.~He, F.~Schremmer, and Q.~Yu,
\emph{Cohen--Macaulayness of local models via shellability of the admissible set},
preprint, arXiv:2603.05875 (2026).

\bibitem{HY}
X.~He and Q.~Yu,
\emph{Dimension formula for the affine Deligne--Lusztig variety $X(\mu,b)$},
Math. Ann. \textbf{379} (2021), no.~3--4, 1747--1765.

\bibitem{HY24}
X.~He and Q.~Yu,
\emph{Irreducibility of local models},
preprint, arXiv:2412.16575 (2024).

\bibitem{HY25}
X.~He and Q.~Yu,
\emph{Dual shellability of admissible set and Cohen--Macaulayness of local models},
preprint, arXiv:2509.11581 (2025).

\bibitem{kh13}
A.~Khare,
\emph{Weights of simple highest weight modules over a complex semisimple Lie algebra},
preprint, arXiv:1305.4104 (2013).

\bibitem{KhRi12}
A.~Khare and T.~Ridenour,
\emph{Faces of weight polytopes and a generalization of a theorem of Vinberg},
Algebr. Represent. Theory \textbf{15} (2012), no.~3, 593--611.

\bibitem{MAX89}
G.~Maxwell,
\emph{Wythoff's construction for Coxeter groups},
J. Algebra \textbf{123} (1989), no.~2, 351--377.

\bibitem{MV20}
E.~Mili\'cevi\'c and E.~Viehmann,
\emph{Generic Newton points and the Newton poset in Iwahori-double cosets},
Forum Math. Sigma \textbf{8} (2020), Paper No.~e50, 18~pp.

\bibitem{PR26}
G.~Pappas and M.~Rapoport,
\emph{Toric schemes and integral models for Shimura varieties
with $\Gamma_1(p)$-type level},
preprint, arXiv:2602.23245 (2026).

\bibitem{PRS}
G.~Pappas, M.~Rapoport, and B.~Smithling,
\emph{Local models of Shimura varieties, I. Geometry and combinatorics},
in \emph{Handbook of moduli. Vol.~III},
Adv. Lect. Math. (ALM), vol.~26,
Int. Press, Somerville, MA, 2013, 135--217.

\bibitem{PZ13}
G.~Pappas and X.~Zhu,
\emph{Local models of Shimura varieties and a conjecture of Kottwitz},
Invent. Math. \textbf{194} (2013), no.~1, 147--254.

\bibitem{Ren09}
L.~E.~Renner,
\emph{Descent systems for Bruhat posets},
J. Algebraic Combin. \textbf{29} (2009), no.~4, 413--435.

\bibitem{SW20}
P.~Scholze and J.~Weinstein,
\emph{Berkeley lectures on $p$-adic geometry},
Annals of Mathematics Studies, vol.~207,
Princeton University Press, Princeton, NJ, 2020.

\bibitem{Sch24}
F.~Schremmer,
\emph{Affine Bruhat order and Demazure products},
Forum Math. Sigma \textbf{12} (2024), Paper No.~e53, 56~pp.

\bibitem{Vin91}
\`E.~B.~Vinberg,
\emph{On certain commutative subalgebras of a universal enveloping algebra},
Math. USSR-Izv. \textbf{36} (1991), no.~1, 1--22.

\bibitem{Zie95}
G.~M.~Ziegler,
\emph{Lectures on polytopes},
Graduate Texts in Mathematics, vol.~152,
Springer-Verlag, New York, 1995.
Corrected second printing, 1998.

\end{thebibliography}
\end{document}